\documentclass[12pt]{article}
\usepackage{amsmath}
\usepackage{amsthm}
\usepackage{amssymb}
\input color
 
\numberwithin{equation}{section}
\newcommand{\dcb}{\begin{array}{lll}}
\newcommand{\dce}{\end{array}}

\newtheorem{theo}{Theorem}[section]
\newtheorem{prop}[theo]{Proposition}

\theoremstyle{remark}
\newtheorem{rem}{Remark}

\theoremstyle{definition}

\newtheorem{ex}[theo]{Example}
\newtheorem{ass}[theo]{Assumption}

\def\R{\mathbb{R}}
\def\D{\mathbb{D}}
\def\1{\mathbf{1}}
\def \E{\mathbb{E}}
\def \P{\mathbf{P}}
\def \N{\mathbb{N}}
\def \I{\mathbb{I}}
\def\G{\mathtt{G}}

\def\m{m}

\def\P{\mathbb{P}}

\def\EEE{\mathtt{E}}
\def\R{\mathbb{R}}
\def\D{\mathbb{D}}
\def\1{\mathbf{1}}
\def \E{\mathbb{E}}
\def \P{\mathbf{P}}
\def \N{\mathbb{N}}
\def \I{\mathbb{I}}
\def\G{\mathtt{G}}

\def\m{m}

\def\P{\mathbb{P}}

\def\EEE{\mathtt{E}}
\def\DDD{{\cal D}}

\begin{document}

\title{Spectral condition, hitting times and Nash inequality.}

\author{Eva {\sc L\"ocherbach}\footnote{Universit\'e de Cergy-Pontoise, CNRS UMR 8088, D\'epartement de Math\'ematiques,
95 000 Cergy-Pontoise,  France. E-mail: {\tt 
eva.loecherbach@u-cergy.fr} },
Oleg {\sc Loukianov}\footnote{D\'epartement Informatique, IUT de Fontainebleau, Universit\'e Paris Est, route Hurtault, 77300 Fontainebleau, France. E-mail: {\tt oleg.loukianov@u-pec.fr} },
Dasha {\sc Loukianova}\footnote{D\'epartement de Math\'ematiques, Universit\'e d'Evry-Val d'Essonne, Bd Fran\c{c}ois Mitterrand, 91025 Evry, France. E-mail: {\tt dasha.loukianova@univ-evry.fr}}
}

\maketitle

\def\abstractname{Abstract}

\begin {abstract}
Let  $X$ be a $\mu$-symmetric Hunt process on a {LCCB space} $\mathtt{E}$.  For an open set $\mathtt{G}\subseteq\mathtt{E}$, let $\tau_\mathtt{G}$ be the exit time of $X$ from $\mathtt{G}$ and $A^\mathtt{G}$ be the generator of the process killed when it leaves $\G.$ Let $r:[0,\infty[\to[0,\infty[$ and $R (t) = \int_0^t r(s) ds$.

We give necessary and sufficient conditions for  $\E_{\mu} R (\tau_\G)<\infty$ in terms of the behavior near the origin of the spectral measure of $-A^\G.$ 
When $r(t)=t^l$, $l >  0$, by means of this  condition we derive the Nash inequality for the killed process.

{In the diffusion case this  permits to show that the existence of moments of order $l+1$ for {$\tau_\G$} implies the Nash inequality {of order $p=\frac{l+2}{l+1}$} for the whole process. The {associated} rate of convergence of the semi-group in {$\mathbb{L}^2(\mu)$} is bounded by $t^{-(l+1)}$.}

Finally, we show for {general Hunt processes  that the Nash inequality giving rise to a convergence rate of order $t^{-(l+1)}$ of the semi-group} implies the existence of moments of order $l+1-\varepsilon$ {for $\tau_\G$}, for all $ \varepsilon>0$.
\end{abstract}



\def\abstractname{Abstract}

\section{Introduction}

In the recent literature on convergence rates for continuous time Markov processes, the link between functional inequalities and the integrability of hitting times has regained a new interest.

The most studied case is undoubtedly the exponential one. It is known since Carmona-Klein \cite{CaKl} (1983), that for a very general  Markov process with invariant probability $\mu$ and Dirichlet form $({\cal E},{\cal D}({\cal E}))$ on $\mathbb{L}^2(\mu)$, the Poincar\'e inequality 
\begin{equation*}
\mu(f^2)\leq C_P{\cal E}(f,f),\quad f\in{\cal D}({\cal E}),\quad \mu(f)=0, 
\end{equation*}
implies the exponential $\mu$-integrability of hitting times of open sets. The converse implication for reversible diffusions can be deduced from the Down-Meyn-Tweedie work \cite{DMT} (1995) on exponential convergence to equilibrium. In the particular case of  linear diffusions, a simple proof of the equivalence between Poincar{\'e} inequality and exponential integrability of hitting times, with explicit estimations, was given in Loukianov, Loukianova and Song \cite{LLS} (2011). In a recent preprint \cite{CGZ} by Cattiaux, Guillin and Zitt (2011), the authors show that for symmetric hypo-elliptic diffusions in $\R^n$, both are equivalent to the existence of Lyapounov functions.


Although the exponential case, at least for diffusion processes, is now fairly well understood, the sub-exponential case, and in particular the polynomial one, is less studied. To the best of our knowledge, the first work in this direction was done by Mathieu \cite {Mat} (1997). For  a diffusion driven by a polynomially decreasing potential, he gives a bound for the first moment of hitting times and relates this bound to some functional inequality.

More recently, the last chapter of  \cite{CGZ} is devoted to the study of the polynomial case. For uniformly strongly hypo-elliptic symmetric diffusions on $\R^n$,  using  Lyapounov functions, the authors show that for open $U$ the finiteness of polynomial moments of hitting times $v_m(x)=\E_x(T_U^m)$, $m\in\N $, together with a local Poincar{\'e} inequality (see \cite{CaGlyap}) implies the weak Poincar{\'e} inequality 
\begin{equation}\label{eq:wp}
\mu(f^2)\leq \beta(s){\cal E}(f,f)+sOsc(f)^2,\quad s>0, \quad f\in{\cal D}({\cal E}),\quad \mu(f)=0 ,
\end{equation}
with the rate-generating function $\beta$ of the  form 
\begin{equation}\label{eq:beta}
\beta(s)=C\left(\inf\left\{u :  \mu\left(\frac{v_{m-1}}{1+v_m}<u\right)>s\right\}\right)^{-1}.
\end{equation}
 
It is well known since the work of Liggett \cite {Lig} (1991), R\"ockner and Wang \cite{RoWa} (2001) and Wang \cite{Wa} (2003) that the weak Poincar{\'e} inequality~\eqref{eq:wp} gives rise to the $\mathbb{L}^2-$convergence of the semigroup with the speed at least
$$ \xi(t):=\inf\{s>0;\ -(1/2)\beta(s)\log s\leq t\}.$$
When the weak Poincar\'e inequality is deduced as a consequence of the finiteness  of the $m$-th moment $v_m(x)=\E_x(T_U^m)$, one interesting question is the explicit dependence of $\xi(t)$ on $m$.  Unfortunately, the implicit form of $\beta(s)$ in \eqref{eq:beta} makes it difficult to obtain this dependence explicitly.
 

The aim of the present work is to describe more explicitly an inequality which corresponds to the finiteness of polynomial moments of hitting times.

It is known that in the case $\beta(s)=cs^{1-p}$ with  $p>1$ and  some $c>0$, the weak Poincar\'e inequality \eqref{eq:wp} is equivalent to the following Nash inequality of order $p$:
 \begin{equation}\label{eq:n}
 \mu(f^2)\leq C{\cal E}^{1/p}(f,f)\Phi^{1/q}(f), \quad f\in{\cal D}({\cal E}), \quad \mu(f)=0, \quad \frac1p+\frac1q=1,
 \end{equation}
 where $\Phi(f)=Osc(f)$ and $C>0.$ \\
Hence in this paper we concentrate on the study of the Nash inequality. More precisely, we show that the finiteness of polynomial (not necessarily integer) moments of hitting times is related to the  Nash inequality with explicit relation between the order of the moment, the order of the inequality and the speed of  convergence of the semigroup. Let $l > 0.$  Our result can be summarized in the following scheme:
 \begin{equation}\label{eq:scheme}
 \E_{\mu}T_U^{l+1}<\infty\Longrightarrow \mbox{Nash inequality of order $\frac{l+2}{l+1}$ } \\
   \Longrightarrow\ \E_{\mu}T_U^{l+1-\varepsilon}<\infty,
 \end{equation}
 for all $ \varepsilon>0.$
 Moreover it is well known since \cite {Lig} that for symmetric semigroups, the Nash inequality of order $\frac{l+2}{l+1}$ is equivalent to  
$$\mu ((P_t f)^2)\le C \Phi (f){t^{-(l+1)} }, \mu (f) = 0 , f \in \mathbb{L}^2 (\mu) .$$ 
The first implication of (\ref{eq:scheme}) is proved only in the diffusion case, but the second one is valuable for a very general Markov process. The method to prove the first implication relies on the use of killed processes. More precisely, we establish a condition for the existence of general hitting time moments in terms of spectral properties of the killed process. This spectral condition
generalizes the  well known equivalence ``exponential moments $\Longleftrightarrow$ spectral gap''.

Let us now give {the} precise statement of our results. $X$ will be a $\mu$-symmetric Hunt process on a {LCCB space} $\mathtt{E}$ where $\mu$ is a bounded Radon measure (wlog we suppose that $\mu $ is a probability measure). For an open set $\mathtt{G}\subseteq\mathtt{E}$, set
$
\tau_\mathtt{G}=\inf\{t\geq 
0: X_t\notin \mathtt{G}\}
$
the exit time of $X$ from $\mathtt{G}$ and put
$
P^\mathtt{G}_t[\mathtt{A}](x)=\mathbb{P}_x[X_t\in \mathtt{A};t<\tau_\mathtt{G}]
$
for {a} measurable subset $\mathtt{A}$ of $\mathtt{E}$.
Denote $A^\mathtt{G}$  the infinitesimal generator of $(P_t^\mathtt{G})$ in $\mathbb{L}^2(\mathbb{I}_\mathtt{G}\cdot \mu (dx))$
and let $(E_\xi^{\G},\xi\geq 0)$ be its spectral family.
 
It is known, see e.g. Friedman \cite {Fri} (1973) or Loukianova, Loukianov and Song \cite {LLS} (2011), that  $\E_{\mu}\exp(\lambda\tau_\G)<\infty$; $\lambda<\lambda_0,$ is equivalent to the fact that {  $-A^\G$ has a spectral gap of width at least equal to $\lambda_0.$ It turns out that hitting time moments generated by other functions than the exponential ones are still related with the spectral properties of $-A^\G$ in the following sense: 
Let $r:[0,\infty[\to[0,\infty[ ,$ $R(t)=\int_0^tr(s)ds$
and denote by
$\Lambda_r:[0,\infty[\to[0,\infty]$ the Laplace transform of $r$:
\begin{equation}\label{firstspecond}
\forall \xi \geq 0,\quad\Lambda_r(\xi)=\int_0^{\infty}r(t)e^{-\xi t}dt.
\end{equation}
We show in Theorem \ref{theo: spectralcondition} that $\E_{\mu}R(\tau_\G)<\infty$ if and only if the spectral measure of $-A^\G$ integrates $\Lambda_r:$  
$$\forall f:\G\to\R,\   \|f\|_{\infty}<\infty,\quad\int_{[0,\infty[}\Lambda_r(\xi)d(E_{\xi}^{\G}f,f)<\infty.$$ 
This condition on the spectral measure will be called in the sequel the {\it $r$-spectral condition. } Then we show how we can derive in a very elementary way the Nash inequality for the killed process $X^\G$ with the help of the spectral condition specified by $r(t)=t^l$ (Proposition \ref{killednash}). In this case the corresponding  rate of transience of the killed process, i.e. the rate of convergence of $P_t^\G$ to zero, is given by $t^{-(l+1)}$.  All this is the content of Section \ref{Section2}, which is entirely devoted to the study of the killed process.

In Section $3$ we address the question how the polynomial spectral condition for the killed process (equivalently the existence of polynomial moments of hitting times) can be used to derive the Nash inequality for the non-killed process.
In this section, our method applies only in the case when the Dirichlet form is local, i.e. in the diffusion case, in the sense that $X$ has a.s. continuous trajectories.
But we do not need to suppose that the process is driven by a stochastic differential equation. 

In the one-dimensional case, from the existence of polynomial moments of order $l+1 > 1$, we derive the Nash inequality specified by $p=\frac{l+2}{l+1}$  without any further assumptions. 

The multidimensional diffusion case is treated as well. Here we need  an additional non-degeneracy condition on the diffusion:  like in \cite{CGZ}, we have to suppose that a local Poincar\'e inequality on some small domain holds, see Remark {\ref{tobecited}.}  At the end of this section we provide the example of a multidimensional diffusion verifying H\" ormander's condition for which our result  holds.

Finally, in Section $4$ we study the implication ``Nash inequality $\Longrightarrow$ polynomial moments''.  The Nash inequality gives an explicit $\alpha$-mixing rate of the process, and then the main idea is to use this mixing rate in order to obtain a deviation inequality to estimate $\P_\mu(\tau_\G>t).$ This nice idea is borrowed from Cattiaux and Guillin (2008), \cite{CG}. 
As a consequence, Nash inequality of order $p=\frac{l+2}{l+1}$ implies the existence of the polynomial moments of hitting times of order $l+1-\varepsilon$, for any $\varepsilon>0$. Note also that this last section is valuable for general Hunt processes. 
     
\section {Killed process.}\label{Section2}

\subsection {Modulated moments and spectral condition for the killed process.  }\label{sec:moments-condition}

Consider a Hunt process $X$ on a {LCCB space} $\mathtt{E}$ in the sense of Fukushima, Oshima, Takeda ~{(1994)}, \cite{FOT}. Let $\mu$ be a Radon measure on $\mathtt{E}$. Suppose that $\mu$ is bounded (wlog $\mu$ is supposed to be a probability measure) and that $X$ is a $\mu$-symmetric process. 
Let $(P_t)_{t\geq 0}$ be the transition semigroup of $X$. Denote by $\mathbb{P}_x$ the law of the process $X$ issued from $x\in\mathtt{E}$.

For an open set $\mathtt{G}\subseteq\mathtt{E}$, set
$$
\tau_\mathtt{G}=\inf\{t\geq 
0: X_t\notin \mathtt{G}\}
$$
the exit time of $X$ from $\mathtt{G}$. All the long of this section we suppose $\tau_\G<\infty$ almost surely. Introduce 
$$
P^\mathtt{G}_t[\mathtt{A}](x)=\mathbb{P}_x[X_t\in \mathtt{A};t<\tau_\mathtt{G}]
$$
for {a} measurable subset $\mathtt{A}$ of $\mathtt{E}$, and set
$$
X^\G_t=\left\{
\dcb
X_t,&&0\leq t <\tau_\mathtt{G}\\
\Delta && t\geq \tau_\mathtt{G} .
\dce
\right.  
$$
Then, according to~\cite{FOT}, $X^\G$ is a Hunt process on the state space $\mathtt{G}\cup\Delta$, symmetric with respect to the measure $\mathbb{I}_\mathtt{G}\cdot \mu (dx),$ with transition semi-group $(P_t^\mathtt{G})_{t \geq 0}$. If $A^\mathtt{G}$ denotes the infinitesimal generator of $(P_t^\mathtt{G})$ in $\mathbb{L}^2(\mathbb{I}_\mathtt{G}\cdot \mu(dx))$, $A^\mathtt{G}$ is a self-adjoint negative operator. Let us denote by $(\cdot,\cdot)$ the scalar product in $\mathbb{L}^2(\mathbb{I}_\mathtt{G}\cdot \mu (dx))$ and by $(E_\xi^{\G},\xi\geq 0)$ the spectral family of $-A^\mathtt{G}$. 

Recall now the basic properties of the spectral decomposition.
$(E^{\G}_\xi,\xi\geq 0)$ is a right-continuous and increasing family of projection {operators} such that for any bounded and continuous {function $f$ defined on} $[0,\infty[, $ $f(-A^\G)$ is given by
\[f(-A^\G)u=\int_{[0,\infty[}f(\xi)\,d E_\xi^{\G} u,\quad u\in\mathbb{L}^2(\mathbb{I}_\mathtt{G}\cdot \mu (dx)).\]
In particular,
\[Id\  u=\int_{[0,\infty[}\,d E_\xi^{\G} u, \quad\quad-A^\G u=\int_{[0,\infty[}\xi\,d E_\xi^{\G} u \] and

\[P_t^\G u=\exp({t A^\G}) u=\int_{[0,\infty[}e^{-\xi t}\,d E_\xi^{\G} u .\]
Recall that for all {$ u,v\in \mathbb{L}^2(\mathbb{I}_\mathtt{G}\cdot \mu (dx)),$}
\[(f(-A^\G)u,g(-A^\G)v)=\int_{[0,\infty[}f(\xi)g(\xi)\,d(E_\xi^{\G} u,v).\]
Actually,  the bounded variation  function $\xi\to(E_\xi^\G u,u)$ is only increasing on the spectrum of $-A^\G$ and its discontinuity points are eigenvalues   
  of $-A^\G.$
Denote by $\mathcal{E}_\G$ the Dirichlet form associated with $-A^\G$ on $\mathbb{L}^2(\mathbb{I}_\mathtt{G}\cdot \mu (dx))$. We have
$$\mathcal{E}_\G(u,v)=\int_{[0,\infty[}\xi\,d(E_\xi^\G u,v).$$
Let $H^\G_\xi$ be the image space of $E^\G_\xi$.
 \begin{prop}
Under the condition $\tau_\G<\infty$ almost surely, we have $H^{\G}_0=\{0\}.$
 \end{prop}
 \begin{proof}
  $H_0^{\G}$ is invariant under $P^\G_t$ for all $t>0$. Indeed, using $E^{\G}_{\lambda} E^{\G}_0=E^{\G}_0$ $\forall \lambda\geq 0$ we see that $\forall u\in H^{\G}_0,$ $\forall t\geq 0$
 \[P_t^\G u=\int_{[0,\infty[}e^{-\xi t}\,d E_\xi^\G u=e^0E_0^\G u=u .\]
 For all $v\geq 0, $ bounded, $\lim_{t\to\infty}  \,  P_t^\G v=\lim_{t\to\infty}\E[v(X_t)\1 _{t<\tau_\G}]=0$ and hence
 $(u,v)=(P_t^\G u,v)=(u,P_t^\G v)\to 0,\ t\to\infty.$ Positive bounded {functions being} dense in {$\mathbb{L}^2$}, we conclude that $u=0.$
 \end{proof}
Since $E^\G_0=0$, one has $\int_{[0,\infty[}f(\xi)\,dE_\xi^\G u = \int_{]0,\infty[}f(\xi)\,dE_\xi^\G u$, so the integral makes sense even if $f$ is not defined at $0$. 
 
It is known, see \cite {Fri} and \cite {LLS}, that the existence of exponential moments of $\tau_\G$ is equivalent to the fact that {  $-A^\G$ has a spectral gap or, equivalently, $\int_{]0,\lambda_0[}dE_{\xi}^\G=0$ for some $\lambda_0>0.$
It turns out that moments generated by other functions than the exponential ones, are still related with the spectral properties of $-A_\G$. In this section we  give necessary and sufficient conditions for the existence of arbitrary moments of $\tau_\G$ in terms of the behavior near the origin of the spectral measure $dE_\xi^\G$.

Let $r:[0,+\infty[\to[0,+\infty[$ be some measurable non-decreasing function, and denote $\Lambda_r: [0,\infty[\to[0,\infty]$ its Laplace transform:
\begin{equation}\label{eq:lapl}
\forall \xi \geq 0,\quad\Lambda_r(\xi)=\int_0^{\infty}r(t)e^{-\xi t}dt.
\end{equation}
Instead of hitting {time} moments, we consider more generally modulated moments defined by $\int_0^{\tau_\G}r(t)f(X_t)dt.$ Denote by ${\cal{B}}_b$ the space of  Borel-measurable and bounded real functions. 
Let $R(t)=\int_0^tr(s)ds$  and    $\| .\|_1:=\| .\|_{L^1(\I_\G \cdot \mu (dx))}.$

\begin{theo}\label{theo: spectralcondition}
  The following four conditions are equivalent:
\begin{enumerate}
\item 
$\E_{\mu}R({\tau_\G})<\infty;$ 
\item For all $f\in {\cal{B}}_b,$ 
$x\to f(x)\times \E_x\int_0^{\tau_\G}r(t)f(X_t)dt\in L^1(\I_\G \cdot \mu (dx));$
\item For all $f\in {\cal{B}}_b,$ 
$\quad\int_{[0,\infty[}\Lambda_r(\xi)d(E_{\xi}^\G f,f)<\infty;$

\item  For all $f\in {\cal{B}}_b,$ 
$\quad\int_0^\infty r(t) \|P_{t/2}^\G f\|^2\, \,dt<\infty.$

\end{enumerate}
Moreover, {for any  $f \in {\cal{B}}_b$,}
\begin{equation}\label{eq:equality}
\|f \times \E_.\int_0^{\tau_\G}r(t)f(X_t)dt\|_{1}=
\int_{[0,\infty[}\Lambda_r(\xi)d(E_{\xi}^\G f,f)=\int_0^\infty r(t) \|P_{t/2}^\G f\|^2\, \,dt .
\end{equation}
\end{theo}

\begin{rem}
In the sequel  the  condition $3.$ of Theorem \ref {theo: spectralcondition}  will be called the{ \bf{\it $r$- spectral condition {for the killed process.}}}
\end{rem}


\begin{proof}
The equivalence $1\iff 2$ is obvious.
The following calculus yields $2. \iff 3. \iff 4. $ and the equality \eqref{eq:equality} for positive bounded functions.
\begin{multline*}
\left(f,\E _{.} \int_0^{\tau_\G} r(t)f(X_t)\,dt\right ) = \left (f,\int_0^\infty r(t) P_t^\G f(.)\,dt\right ) = \int_0^\infty r(t)\left (f,P_t^\G f\right)\, \,dt\\
=\int_0^\infty r(t) \|P_{t/2}^\G f\|^2\, \,dt
=\int_0^\infty r(t)\left (f,\int_{[0,\infty[}e^{-\xi t}dE_\xi^\G f\right)\, \,dt=\\
=\int_0^\infty r(t)\int_{[0,\infty[} e^{-\xi t}\, d\left (E_\xi^\G f,f\right)\, \,dt
= \int_{[0,\infty[} \int_0^\infty r(t) e^{-\xi t}\,dt\,d\left (E_\xi^\G f,f\right) =\\=\int_{[0,\infty[}\Lambda_{r}(\xi)\,d\left(E_\xi^\G f,f\right). 
\end{multline*}
If $\E_{\mu} R({\tau_\G})<\infty$ and $f$ is bounded, to show the equalities of \eqref{eq:equality} we use
$$\E_x\int_0^{\infty} r(t)|f(X_t)|\1_{(t<\tau_{\G})}dt\leq \|f\|_{\infty}\int_0^{\infty}r(t)\P_x(\tau_{\G}>t)dt=\|f\|_{\infty}\E_xR(\tau_{\G}),$$
the last expectation being finite for $\mu $-almost all $x\in \G.$
\end{proof}
\begin{ex}
Consider the case $r(t)=t^l,\quad l >  0$.
 We have for $\xi \geq 0$
$\Lambda_r(\xi)=\Gamma(l+1)\xi^{-(l+1)}.$ Hence 
\begin{equation}\label{eq:polynomspectr}
x\to\E_x\tau_{\G}^{l+1}\in L(\I_\G \cdot \mu (dx))\quad \Longleftrightarrow\quad\int_{[0,\infty[}\xi^{-(l+1)}d(E^\G_{\xi}f,f)<\infty , 
\end{equation}  
for all $f$ non-negative and bounded.  In the next section we will explain how to use the spectral condition to obtain functional inequalities for $X^{\G}$ and then for $X.$ 
\end{ex}
\begin{ex}
Consider the case $r(t)=e^{\lambda t},\quad \lambda >0.$ We have
$$\Lambda_r(\xi)=\frac1{\xi-\lambda},\quad\mbox{ if }\quad\xi >\lambda,\quad \Lambda_r(\xi)=+\infty\quad \mbox{ otherwise. }$$
Put \[\lambda_0=\sup\{\lambda>0,\quad x\to \E_xe^{\lambda\tau_{\G}}\in L^1(\I_\G \cdot  \mu (dx))\} .\]
We obtain that $\lambda_0$ is the infinum of the spectrum of $-A^\G.$ 
\end{ex}

\subsection{  Polynomial spectral condition and Nash inequality for killed process.}

In \cite {Lig} Liggett introduced the following Nash inequality for a Dirichlet form $\mathcal{E}(f,f)$ associated to a linear operator generating a strongly continuous Markovian semigroup with invariant {probability} measure $\mu .$ 
\begin{equation}\label {eq: nash}
\mu ({( f-\mu (f))}^2)\leq C\mathcal{E}^{1/p}( f, f)\Phi^{1/q}( f), \quad   f \in D ( \mathcal{E}). 
\end{equation}
Here $1< p,q < \infty$ with $1/p+1/q=1,$ $C$ is a positive constant, and {$\Phi: \mathbb{L}^2(\mu )\to[0,\infty]$} satisfies $\Phi(cf)=c^2\Phi(f),$ for any $c\in\R$ and {$f\in \mathbb{L}^2(\mu ).$}
It is shown in \cite{Lig} that if in addition $\Phi(P_tf)\leq \Phi( f)$ {$\forall f\in \mathbb{L}^2(\mu),$ $\forall t>0,$} then the inequality \eqref{eq: nash} is equivalent to 
\begin{equation}\label{eq: sgconv}
\exists C>0,\quad\|P_t(f)-\mu(f)\|^2_2\leq C\frac{\Phi (f)}{t^{q-1}}
\end{equation}
for all $ f \in \mathbb{L}^2(\mu)$ and $ t>0.$
If the semi-group of $X$ is conservative, symmetric and ergodic,  $\mu (f)=E_0f.$ 
Hence we will consider {the following form of the Nash inequality:}
\begin{equation}\label{eq:vrainash}
\|f-E_0(f)\|^2\leq C\mathcal{E}^{1/p}( f, f)\Phi^{1/q}( f), \quad   f \in D ( \mathcal{E}). 
\end{equation}   
Let us point out again that for the killed process the semi-group is not conservative, transient, and $E_0=0.$ The following proposition shows that the condition  $\E_{\mu} \tau_\G^{l+1}<\infty$ implies {the} Nash inequality in the form \eqref{eq:vrainash} for the killed process.



\begin{prop}\label{killednash}
Let $l > 0$.
Suppose that $\E_{\mu}\tau_\G^{l+1}<\infty.$
Then the Nash inequality {\eqref{eq:vrainash}}
 holds {for the killed process}  with $p=\frac{l+2}{l+1}$ and $q=l+2$ and 
 $\Phi(f)\leq \|f\|_{\infty}^2\E_{\mu}\tau_\G^{l+1}.$
\end{prop}
\begin{proof}
In virtue of {Theorem \ref{theo: spectralcondition}}, the condition $\E_{\mu}\tau_\G^{l+1}<\infty$ is equivalent to 
$$\int_{[0,\infty[}\frac 1{\xi^{l+1}}d(E_{\xi}^\G f,f)<\infty,$$
for all  bounded $f .$ Let  $f \in D (\mathcal{E}_\G).$  Suppose that $p^{-1}+q^{-1}=1$ and write, using H\"older's inequality:
\begin{multline*}
\|f-E_0 f\|^2=\|f\|^2=\int_{[0,\infty[}d(E_\xi^\G f,f)=\int_{[0,\infty[}\xi^{1/p}\xi^{-1/p} d(E_\xi^\G f,f)\leq\\
\left(\int_{[0,\infty[}\xi d(E_\xi^\G f,f)\right)^{1/p}\times \left(\int_{[0,\infty[}\xi^{-q/p} d(E_\xi^\G f,f)\right)^{1/q}
= \mathcal{E}_\G^{1/p}(f)\Phi^{1/q}(f),
\end{multline*}
where
\[\Phi(f)=\int_{[0,\infty[}\xi^{-q/p} d(E_\xi^\G f,f).\]
Now we choose $p$ and $q$ in such a way that 
$$ \Phi(f)=\int_{[0,\infty[}\xi^{-(l+1)} d(E_\xi^\G f,f).$$

This choice is given by $p=\frac{l+2}{l+1}$ and $q=l+2.$
Finally we obtain for all  $f \in D (\mathcal{E}_\G)$ 
{\begin{equation*}
\|f-E_0 f\|^2\leq
\mathcal{E}_\G(f)^{{(l+1)}/{(l+2)}}\times\Phi^{1/(l+2)}(f) ,
\end {equation*}
}where $\Phi$ satisfies $\Phi(cf)=c^2\Phi(f)$ for any $c\in\R$ 
and $\Phi(f)<\infty$ for  all bounded $f$.
Also
\[\Phi(P_t^\G f)=\int_{[0,\infty[}\xi^{- (l+1) }e^{-2 \xi t} d(E_\xi^\G f,f)\leq \Phi(f).\]

\end{proof}

\section{Polynomial moments and Nash inequality for non-killed process.}\label{Section3}
In this section we show how polynomial modulated moments are related to Nash inequality for the non-killed process. The result can be resumed as follows. For all $ l > 0,$ for all $ \varepsilon >0$ we have: 
``integrability of  moments of order $l+1$ $\Longrightarrow$ Nash inequality giving rise to {$\mathbb{L}^2$} convergence of the semigroup with speed $t^{-(l+1)}$$\Longrightarrow$ existence of  moments of order $l+1-\varepsilon$".   
 
For the second implication we work under the general conditions of Section \ref{Section2}.
For the first implication ``moments imply Nash'' we work in the diffusion case only. In dimension $1,$ no hypothesis on the diffusion is imposed. In higher dimension, however, we need a non-degeneracy condition which is a local Poincar\'e inequality (see the comments in Remark \ref{tobecited}). 
\subsection{Polynomial moments $\Longrightarrow$ Nash inequality. One-dimensional diffusion case.}

{In this subsection we  show that the Nash inequality for a killed diffusion process on $\R$ implies the Nash inequality for the non-killed process. Fix some $a \in \R$ and let $\G^-=]-\infty,a[$, $\G^+=]a,\infty[$. 
We use some well-known techniques 
which are specific to the one-dimensional case.}

Since $X$ is a diffusion, it possesses a scale function $S$ and a corresponding speed measure $m.$ Denote by $dS$ the measure induced by $S(x)$. Let $F(x)$ be a real function on $\R$. We shall write $dF\ll dS$, if there exists a function $f(x)$ in $\mathbb{L}_{loc}^1(dS)$ such that 
$$\int_a^bf(x) dS(x) = F(b)-F(a), \;  \forall a<b.$$
The function $f(x)$ will be denoted $\frac{dF}{dS}(x)$. Introduce then the function spaces
\begin{eqnarray}\label{defF}
\mathcal{F}=\left\{F\in \mbox{{ $\mathbb{L}^2(\m)$}}: \quad dF\ll dS, \ \frac {dF}{dS}\in \mbox{ { $ \mathbb{L}^2(dS)$} }\right\},\\
\mathcal{F}_{]a,\infty[}=\{F\in\mathcal{F}: F(x)=0, \ x\leq a\},\nonumber\\
\mathcal{F}_{]-\infty,a[}=\{F\in\mathcal{F}: F(x)=0, \ x\geq a\} .\nonumber
\end{eqnarray}
We do not assume that $dS$ and $m$ are absolutely continuous with respect to the Lebesgue measure. We cite the following theorem from \cite{LLS}.

\begin{theo}[\cite{LLS}]\label {D-form of X} The diffusion $X$ is $m$-symmetric. The Dirichlet space associated with $X$ is the function space $\mathcal{F}$ given by~\eqref{defF}, and the Dirichlet form has the expression 
\[
\mathcal{E}(F,F)=\int_{-\infty}^\infty \left(\frac{dF}{dS}\right)^2(x)dS(x), \ F\in\mathcal{F}.
\]

The restriction of the Dirichlet form $\mathcal{E}$ on $\mathcal{F}_{]a,\infty[}$ is the Dirichlet form {$\mathcal{E}_{]a,\infty[}$} associated with the semigroup $(P^{]a,\infty[}_t)_{t\geq 0}$ of the process $X$ killed when it exits $]a,\infty[$. The killed process $X^{]a,\infty[}$ is symmetric with respect to the measure $\mathbb{I}_{]a,\infty[}\cdot m(dx)$.

The same is true (with obvious modifications) for {$\mathcal{E}_{]-\infty,a[}$.}
\end{theo}
The proof of this theorem is given in \cite {LLS}.\\
We can now state the Nash inequality for the non-killed process $X.$
For $a\in\R$ introduce the hitting time $T_a=\inf\{t\geq 0, \ X_t=a\}.$  
\begin{theo}\label{theo:nashnonkilled}
Let $l > 0$.  Suppose that for some $a\in\R$ 
\begin{equation}\label{eq:lower}
\int_{-\infty}^{+\infty}\E_x( T_a)^{l+1}  m(dx)<\infty .
\end{equation}
Then {for $ \mu (.) = \frac{1}{m (\R)} m (.),$}  Nash inequality 
$$\mu({( F-\mu(F))}^2)\leq 
\mathcal{E}^{1/p}( F, F)\Phi^{1/q}( F), \quad   F \in D({\cal E}) $$
 holds with $p=\frac{l+2}{l+1},$ $q=l+2 .$ The function {$\Phi : \mathbb{L}^2 (\mu) \to [0, +\infty]  $} satisfies   
 $\Phi(cF  )=c^2\Phi(F)$ for all $c\in\R$ and  $\Phi (P_t F ) = \Phi (F )$ for all $ t>0.$
 Also for all {$ F \in \mathbb{L}^2 (\mu ),$}
\begin{equation}\label{eq:last}
 \Phi (F) \le C \| F - \mu(F) \|_\infty^2 . 
\end{equation}   
\end{theo}

\begin{rem}
Note that by the ``all-or-none'' property obtained in \cite{LL}, Theorem 4.5, (\ref{eq:lower}) holds for some $a$ if and only if it
holds simultaneously for all $ a \in \R.$  { In this case} 
 $\E_xT_a<\infty$ $\forall x\in\R,\forall a\in\R ,$ and hence $m(\R)<\infty.$
\end{rem}
{\begin{rem} Note also that 
\begin{equation}\label{eq:osc}
\Phi (F) \le C(\sup_R F-\inf_R F)^2=C \; {Osc} (F)^2.
\end{equation}
\end{rem}}

\begin{proof}
Fix a point $a \in \R .$ Then  
the variational formula for the variance gives for all $F\in\mathcal{F},$ 
\begin {multline*}
\int_{-\infty}^{+\infty}\left (F(x)-\mu (F)\right )^2 dm(x)
 \leq \int_{-\infty}^{+\infty}\left ( F(x)- F(a)\right )^2 dm(x)=\\
= \int_{-\infty}^{a}\left ( F(x)- F(a)\right )^2 dm(x)+\int_{a}^{+\infty}\left ( F(x)- F(a)\right )^2 dm(x).
\end{multline*}
Write
$$ F_- (x)= ( F(x) -  F(a))1_{\{ x < a \}} ,\quad \; F_+ (x)= ( F (x) -  F(a)) 1_{\{x > a\} } .$$
Then $F_- \in {\cal F}_{ ]-\infty, a[} $ and $ F_+ \in {\cal F}_{ ]a, \infty [}.$ Hence we can 
apply Proposition \ref{killednash} for both $\G^-=]-\infty,a[$, $\G^+=]a,\infty[$.
Denote $$\mathcal{E}_{]-\infty,a[}=\mathcal{E}_-\quad \mbox{ and}\quad \mathcal{E}_{]a, +\infty[}=\mathcal{E}_+.$$
Denote
$$ \Phi_-(u)=\int_{[0,\infty[}\xi^{-(l+1)} d(E^{\G^-}_\xi u,u),\quad\mbox{and}\quad \Phi_+(u)=\int_{[0,\infty[}\xi^{-(l+1)} d(E^{\G^+}_\xi u,u),$$
then with $p=\frac{l+2}{l+1}$ and $q=l+2,$
\begin{multline*}
 \int_{-\infty}^{a}\left (F(x)- F(a)\right )^2 dm(x)+\int_{a}^{+\infty}\left ( F(x)- F(a)\right )^2 dm(x)\leq\\
\leq \mathcal{E}_-^{1/p}(F_-)\Phi_-^{1/q}(F_-)+\mathcal{E}_+^{1/p}(F_+)\Phi_+^{1/q}(F_+) \leq \\
\leq\Phi_-^{1/q}(F_-) \left(\int_{-\infty}^a\left (\frac {dF}{dS}\right)^2(t)dS(t)\right)^{1/p}\\
+\Phi_+^{1/q}(F_+) \left(\int_a^{+\infty}\left (\frac {dF}{dS}\right)^2(t)dS(t)\right)^{1/p} \leq\\
\leq  (\Phi_-^{1/q}(F_-)+\Phi_+^{1/q}(F_+)) \left(\int_{-\infty}^{+\infty}\left (\frac {dF}{dS}\right)^2(t)dS(t)\right)^{1/p}=\mathcal{E}^{1/p}(F)\Phi_a^{1/q}(F) ,
\end {multline*}
where 
$$ \Phi_a(F) =(\Phi_-^{1/q}(F_-)+\Phi_+^{1/q}(F_+))^q .$$ 
The above result holds for any $ a \in \R .$ Hence we can put  
\begin{equation} \Phi (F) = \sup_{ t \geq 0} \inf_{ a \in \R} \Phi_a ( P_t F) .\end{equation}
Then 
$$ \Phi ( cF ) = c^2 \Phi(F) \mbox{ and }  \Phi ( P_t F ) = \Phi (F)$$
are trivially satisfied. 
It remains to show that under the conditions of the theorem, $\Phi $ satisfies (\ref{eq:last}).

In virtue of Theorem \ref{theo: spectralcondition},
\begin{multline}
\Phi_-(F_-)=\int_{-\infty}^a(F(x)-F(a))\times\E_x\int_0^{T_a}s^l\times(F(X_s)-F(a))ds\ m(dx)\leq\\
 \leq 4\|F-\mu(F)\|^2_{\infty}\int_{-\infty}^a\E_xT_a^{l+1}m(dx) .
\end{multline}
{In the same way,}
$$\Phi_+(F_+)\leq4\|F-\mu(F)\|^2_{\infty}\int_a^{+\infty}\E_xT_a^{l+1}m(dx)$$
and {thus}
\begin{multline*}
\Phi_a(F)\leq4\|F-\mu(F)\|^2_{\infty}\\
\left ((\int_{-\infty}^a\E_xT_a^{l+1}m(dx))^{1/q}+
(\int_a^{+\infty}\E_xT_a^{l+1}m(dx))^{1/q} \right )^q.
\end{multline*}
We deduce, using $\|P_tF\|_{\infty}\leq \|F\|_{\infty}$ and $ \mu (P_t F ) = \mu ( F) $ that 
\begin{multline*}
 \inf_{a } \Phi_a (P_t F) \le 4 \|P_t F - \mu (P_t  F) \|_\infty \;\\
 \inf_{a \in \R} \left ((\int_{-\infty}^a\E_xT_a^{l+1}m(dx))^{1/q}+
(\int_a^{+\infty}\E_xT_a^{l+1}m(dx))^{1/q} \right )^q\\
\le 4 \|  F - \mu ( F) \|_\infty \;\\
\inf_{a \in \R} \left ((\int_{-\infty}^a\E_xT_a^{l+1}m(dx))^{1/q}+
(\int_a^{+\infty}\E_xT_a^{l+1}m(dx))^{1/q} \right )^q ,
\end{multline*}
and this implies (\ref{eq:last}).

\end{proof}

\begin{rem}
Under the assumptions of Theorem \ref{theo:nashnonkilled}, Liggett \cite{Lig}, Theorem 2.2, shows that  
$$ \| P_t F - \mu (F)  \|_2^2 \le C \frac{\Phi (F) }{ t^{ l+1}}  , \; F \in {\cal F} .$$ 
Hence under the assumption of integrability of $l+1-$moments of hitting times we obtain a polynomial decay of the 
transition semigroup $P_t$ of $X$ at the same rate $ t^{ - (l+1)} .$ 
\end{rem}

\subsection{Polynomial moments $\Longrightarrow$ Nash inequality. General diffusion case.}
In this section we come back to the general conditions of Section 2 and consider the {$\mu-$symmetric} Hunt process $X$ on the LCCB space $E$ such that $\mu(E)=1$, with  semigroup $(P_t)_{t\geq 0} $ and associate Dirichlet form $({\cal E},{\cal D}({\cal E}))$ on {$\mathbb{L}^2(\mu).$} 
\begin{ass}\label{regular}
 Assume that
  the Dirichlet space $({\cal E},{\cal D}({\cal E}))$ of $X$
contains 
regularized indicator functions:
For any compact set $K$ and relatively compact open set $G $ with $K \subset G,$ $\exists \; u \in {{\cal D}({\cal E})}$   such that $ 0 \le u \le 1, u \equiv 1 $ on $K,$ $u \equiv 0 $ on $ G^c $.
\end{ass}
This is the case when the Dirichlet form
is regular, {see Fukushima et al. (1994), \cite{FOT}, p.6.}

 
\begin{ass}\label{carreduchamp}
Assume that the Dirichlet form $({\cal E}, {\cal D}({\cal E}))$ 
{admits} a carr\'e du champ  $\Gamma$.
  \end{ass}
{Following Bouleau and Hirsch (1991), \cite {BH}, Proposition 4.1.3., this} means that there exists a unique positive symmetric and continuous bilinear form from ${\cal D}({\cal E})\times {\cal D}({\cal E})$ into {$\mathbb{L}^1(\mu),$} denoted by $\Gamma$ and called {\it the carr{\'e} du champ} operator, such that $\forall f,g,h\in{\cal D}({\cal E}) \cap  { \mathbb{L}^{\infty} }$,
   \begin{equation}\label{eq:defgamma}
    {\cal E}(fh,g)+{\cal E}(gh,f) -{\cal E}(h,fg)=\int \mbox{ { $h \; $}}\Gamma(f,g)d\mu.
   \end{equation}
\begin{ass}\label{locality}
Assume that the Dirichlet form $({\cal E}, {\cal D}({\cal E}))$ is local.
\end{ass}
In this case, {by  \cite {BH}, Proposition $6.1.1.$,}
$$\forall f\in{\cal D}({\cal E}),\quad {\cal E} (f,f) = \frac12 \int_{\EEE} \Gamma (f,f) d \mu .$$ 
Note that the locality of the form is equivalent to assume that the process $X$ is a diffusion process, in the {sense} that $X$ has a.s. continuous {trajectories,} see Theorem 4.5.1 of \cite{FOT}. 
\begin{ass}\label{recurrence}
Assume that ${\cal E}$ is recurrent, i.e. $1\in{\cal D}({\cal E})$ and $\mbox{{ $ {\cal E}(1, 1)$}}=0.$
\end{ass}

Recall the definition of the spaces  $\mathbb D_p$, $p\geq 2$, similar to {the definition of Sobolev spaces, (\cite {BH}, definition $6.2.1$):}
$$\mathbb D_p=\{f\in {\cal D}({\cal E})\cap \mbox{ { $ \mathbb{L}^p;\ \Gamma (f,f)^{1/2}\in \mathbb{L}^p $}}\}.$$
and, for $ f\in {\mathbb D}_p,$
$$\|f\|_{{\mathbb D}_p}=\mbox{ {$ \|f\|_{\mathbb{L}^p}+\|\Gamma (f,f)^{1/2}\|_{\mathbb{L}^P}$}}.$$

The following  proposition is proved in {\cite{BH} (proposition $6.2.3$).}
\begin{prop}\label{prop:623}
Let $p,q,r\geq 2$ with $\frac1p+\frac 1q=\frac1r.$ Then
$$f\in {\mathbb D}_p\quad\mbox{and}\quad g\in {\mathbb D}_q \Longrightarrow fg\in{\mathbb D}_r
\quad\mbox{and}\quad \|fg\|_{{\mathbb D}_r}\leq \|f\|_{{\mathbb D}_p}\|g\|_{{\mathbb D}_q}.$$
\end{prop}
For any set $\G$ and any $r> 0 $ we set {$ \G_r = \{ x : dist (x, \G ) < r \} .$} 
{Under  assumptions \ref {regular}, \ref{carreduchamp}, \ref{recurrence}  the following theorem holds:}  
\begin{theo}\label{theo:oufouf}
Let $l > 0 .$ Suppose there exists an open relatively compact subset $\G \subset \mathtt{E}$ and $r > 0 $ such that the following conditions are satisfied.
\begin{enumerate}
\item 
For $ A \in\{ \G _r; \bar{\G}^c\}$ 
$$ \E_x \tau_A^{l+1} \in \mbox{ { $ \mathbb{L}^1 ( \mu \I_A) $}} .$$ 
\item
$\mu$ satisfies a local Poincar\'e inequality in restriction to 
$ \G_r \setminus \bar \G ,$ that is
$$ \int_{ \G_r \setminus \bar \G} f^2 d\mu \le C_P ( \G , r) \int_{\G_r \setminus \bar \G } \Gamma (f, f) d \mu $$
for all {$f \in {\cal D}({\cal E}) $} having $ \int_{ \G_r \setminus \bar \G} f d \mu = 0 .$ 
\item Suppose that the regularized indicator $u $ of the set $\G$ {given by} $0 \le u \le 1, $ $u = 1 $ on $\bar{\G} $, $u = 0 $ on $\G_r^c ,$ 
verifies  
\begin{equation}\label{eq:constant}
  C(u,r):=\| \Gamma (u,u)  \|_{\infty }<\infty  .
\end{equation}

\end{enumerate}
Then the following Nash inequality holds: For any $ f \in {\cal D}({\cal E}) $ with $\mu (f) = 0,$   
$$ \mu ( f^2 ) \le  {\cal E}^{ 1/p} ( f, f) \Phi^{1/q} (f) ,$$
where $ q = l+2, 1/p + 1/q = 1 .$
Here, the function {$ \Phi : \mathbb{L}^2 (\mu) \to [0, \infty ] $} satisfies $\forall a\in\R,$ $ \Phi (a f ) = a^2 \Phi ( f) $ and $ \Phi ( P_t f ) \le \Phi (f) $ for all $t \geq 0, f\in\L^2(\mu).$
Moreover 
$$ \Phi  (f) \le   2[1  + C(u,r)C_P (\G, r) ] \; { { Osc}} (f)^2[  \int_{\G_r} \E_x \tau_{\G_r} ^{l+1} \mu (dx) +\int_{\bar{\G}^c   } \E_x \tau_{\bar{\G}_c}^{l+1} \mu (dx)  ]. $$

\end{theo}

\begin{proof}
Let {$f\in{\cal D}({\cal E})) $}  with $ \mu (f) = 0 .$ 

By the variational definition of the variance, we have that 
$$ \int  f^2 (x) \mu (dx) \le \int ( \ f (x)- c)^2 \mu (dx) ,$$
where $c$ is given by $ c = \frac{1}{\mu ( \G_r \setminus \bar \G )} 
\int_{ \G_r \setminus \bar \G} f  d \mu .$
The use of this constant will become clear in formula (\ref{eq:Poincareclassique}) later.\
 
Denote $\tilde f =  f -c .$ Let $u $ be the regularized indicator of $\G .$ 
 Write $ \tilde f = {\tilde f} u + {\tilde f} (1- u) .$  Using proposition \ref{prop:623}, since $u\in\D_{\infty}$ and {$\tilde f\in\D_2={\cal D}({\cal E}),$ we have
$\tilde f u\in {\cal D}({\cal E}).$} 
Hence
both $ \tilde f u $ and $ \tilde f (1- u) $ belong to  ${\cal D}({\cal E}).$
Now we can write
\begin{multline}\label{eq:important}
\int_{\mathtt{E}} (\tilde f(x))^2 \mu (dx) = 
\int_{\mathtt{E}}  (\tilde f u + \tilde f (1- u))^2 (x)  \mu (dx)  \\
\le 2 \left[ \int_{\mathtt{E}} (\tilde f u)^2 (x) \mu (dx) + \int_{\mathtt{E}} ( \tilde f (1-u))^2 (x) \mu (dx) \right] \nonumber \\
= 2 \left[  \int_{\G_r} (\tilde f u)^2 (x) \mu (dx) + \int_{\bar{\G}^c} (\tilde f (1-u))^2 (x) \mu (dx) \right] .
\end{multline}
We want to apply Proposition \ref{killednash} both to $ \G_r$ and to $ \bar{\G}^c .$ For that sake, note that 
{$ \tilde f u \in {\cal D}({\cal E}) $} and its quasicontinuous modification is zero on $ \G_r^c .$ Hence by (4.3.1) of \cite{FOT},
$ \tilde f u \in $ {$ {\cal D}({\cal E}_{\G_r}) .$}
In the same way $ \tilde f (1-u) \in$ {$ {\cal D}({\cal E}_{\bar{\G}_c}) .$}  Denote 
$$ \Phi^{\G_r} (\tilde f) = \int_{ [ 0, \infty [} \xi^{ - (l+1)} d ( E_\xi^{ \G_r } \tilde f u , \tilde f u  ) ,$$
and
$$  \Phi^{\bar{\G}^c} (\tilde f)= \int_{ [ 0, \infty [} \xi^{ - (l+1)} d ( E_\xi^{ \bar{\G}^{c}} \mbox{{ $\tilde f$}}  (1-u) , \tilde f (1-u) ).$$
We have
$$ \int_{\G_r} (\tilde f u)^2 (x)  \mu (dx) \le [{\cal E}_{\G_r} ( \tilde f u, \tilde f u )]^{1/p} [\Phi^{\G_r} (\tilde f)]^{1/q} \le 
{\cal E}^{1/p} (\tilde f u , \tilde f u ) [\Phi^{\G_r} (\tilde f)]^{1/q},$$
where the first inequality follows from Proposition \ref{killednash}, and the second since $ {\cal E}_{\G_r } $ is just the restriction of the Dirichlet form $ {\cal E} $ to ${\cal F}_{\G_r } .$
 
Analogously, 
$$ \int_{\bar{\G}^{c}} (\tilde f (1- u))^2 (x)  \mu (dx) \le  {\cal E}^{1/p} (\tilde f (1-u), \tilde f (1-u)) [\Phi^{\bar{\G}^c } (\tilde f)]^{1/q}.$$

We have to control $ {\cal E} (\tilde f u, \tilde f u ) $ and $ {\cal E} ( \tilde f (1-u), \tilde f (1- u) ) .$
We have (Proposition $6.2.3$ of \cite {BH} and {Cauchy-Schwartz)}
\begin{equation}\label{eq: derivation}
\Gamma(\tilde f u,\tilde f u)\leq 2(\Gamma(\tilde f,\tilde f)+\tilde f^2\Gamma(u,u)).
\end{equation}
We need to show that  $ \Gamma ( \tilde f , \tilde f ) = \Gamma (f, f)$ ( equivalently 
$\Gamma(f,c)=0$, $\Gamma(c,c)=0$). 
Note that {by \cite {BH}, prop. $5.1.3.$,} the locality of the form is equivalent to 
$$\forall f,g  \in{\cal D}({\cal E}),\quad \forall a\in\R,\quad
(f+a)g=0\Longrightarrow {\cal E}(f,g)=0.$$

Using the {characterization of $\Gamma$ in \eqref{eq:defgamma}} and the locality of the Dirichlet form we see that 
for any {$h\in{\cal D}({\cal E})\cap \mathbb{L}^{\infty}$}, $\int_E h\Gamma(f,c)d\mu=0$ and $\int_E h\Gamma(c,c)d\mu=0.$ Also, any positive {$h\in \mathbb{L}^1 \cap \mathbb{L}^{\infty}$} is the almost {everywhere} limit of a uniformly bounded sequence of positive elements of {${\cal D}({\cal E})\cap \mathbb{L}^{\infty}$.} Indeed, 
$\lim_{\lambda\to\infty}\lambda R_{\lambda}h=h$ a.s. at least for one subsequence of $\lambda.$ 
For each $\lambda>0$ $\lambda R_{\lambda}h$ belongs to {${\cal D}({\cal E}),$} and if $\|h\|_{\infty}\leq C$ then
$\|\lambda R_{\lambda}h\|_{\infty}\leq C,$ see {\cite{BH}, proposition 4.1.3 and 3.2.1.} Hence, {by dominated convergence,} $\int_E h\Gamma(f,c)d\mu=0$ and $\int_E h\Gamma(c,c)d\mu=0$  for any
 positive {$h\in \mathbb{L}^1 \cap \mathbb{L}^{\infty}.$ But this implies that $ \Gamma (f,c) = 0 $ and $\Gamma (c,c) = 0 $ almost surely. }
  
Still using  locality we show that {$\Gamma (u,u)=0$} on $\bar\G$ and $\G_r^c.$ 
Firstly, using the definition of $\Gamma,$ for any {$f\in{\cal D}({\cal E})\cap \mathbb{L}^{\infty}$,} such that $Supp(f)\in\G_r^c,$
\begin{equation}\int_\E f\mbox{{ $\Gamma(u,u) $ }}d \mbox{{ $\mu$ }}=\mbox{ { $ 2$ }}{\cal E}(fu,u) \mbox{ {$-$} }{\cal E}(u^2,f),
\end{equation}
and since $fu^2=0$, we conclude that  {$\int_E f\Gamma(u,u)d \mu=0$} for such a $f$ and hence forall {$ f\in {\cal D}({\cal E}_{\G_r^c})\cap \mathbb{L}^ {\infty}.$}  
To conclude that {$\Gamma(u,u)=0$ on $\G_r^c$} we need the same property with {$f\in \mathbb{L}^1(\G_r^c)\cap \mathbb{L}^{\infty}.$}  For this sake we use that for any {$f\in \mathbb{L}^1(\G_r^c)\cap \mathbb{L}^{\infty},$} it holds
$$f=\lim_{\lambda\to\infty}\lambda R_{\lambda}^{\G_r^c}f.$$
{along a sub-sequence of $\lambda ,$ almost surely.} 

Note that $ R_{\lambda}^{\G_r^c}f$ has its support in ${\G_r^c},$ $ R_{\lambda}^{\G_r^c}f\in
{\cal D}({\cal E}_{\G_r^c}).$ We approximate $f$ with $\lambda R_{\lambda}^{\G_r^c}f$
and we obtain {$\Gamma (u,u)=0$} on ${\G_r^c}.$
{A similar} argument shows that {$\Gamma (u,u)=0$} on $\G.$

Therefore,
\begin{multline*}
  {\cal E} (\tilde f u, \tilde f u ) 
\le   \int_{ E } \Gamma (f, f) \mu (dx)   
+   \int_{ \G_r \setminus \bar \G}  \mbox{{ $\tilde f^2$}} (x) \Gamma (u,u) (x)  \mu (dx) \\
\le    \int_{ E} \Gamma (f, f) \mu (dx)   
+  C( u, r)  \int_{ \G_r \setminus \bar \G} \tilde f^2 (x) \mu (dx ) ,
\end{multline*}
{which implies that} 
$$ {\cal E} (\tilde f u , \tilde f u ) \le  2 {\cal E} (f , f) +  C(u,r) \int_{ \G_r \setminus \bar \G } \tilde f^2 (x) \mu (dx)  .$$
The role of $u$ and $1 - u$ being symmetric, we get in the same way
$$ {\cal E} ( \tilde f (1- u) , \tilde f ( 1- u)) \le 2 {\cal E} (f ,f) +  C(u,r) \int_{ \G_r \setminus \bar \G } \tilde f^2 (x) \mu (dx)  ,$$
with the same constant $ C (u,r) .$ 
Putting things together, we conclude that

\begin{equation}\label{eq:garnichtschlecht}
\int_{\mathtt{E}} f^2 (x) \mu (dx) \le \left( 2 {\cal E} (f, f)  +  C(u,r) \int_{\G_r \setminus \bar \G } \tilde f^2 (x) \mu (dx)  \right)^{ 1/p} \Psi^{1/q} (  f) ,
\end{equation}
where 
$$ \Psi ( f) = \left( [\Phi^{\G_r } (\tilde f)]^{1/q} + [\Phi^{\bar{\G}^c } (\tilde f)]^{1/q} \right)^q.$$
We have to treat the term 
$$ \int_{\G_r \setminus \bar \G } \tilde f^2 (x) \mu (dx ) .$$
It is here that we need the fact that $\int_{\G_r \setminus \G} \tilde f (x) \mu (dx) = 0 ,$ by definition of the constant $c.$ Now we can apply the local Poincar\'e inequality in order to deduce that 
\begin{equation}\label{eq:Poincareclassique}
\int_{\G_r \setminus \bar \G } \tilde f^2 (x) \mu (dx)  \le C_P (\G ,r)  \int_{ \G_r \setminus \bar \G} \Gamma (f, f) d \mu .
\end{equation}
Coming back to (\ref{eq:garnichtschlecht}) we conclude that 
$$ \int_{\mathtt{E}} f^2 (x) \mu (dx) \le 
\left( [2  + 2 C(u,r)C_P (\G, r) ]  {\cal E} (f, f) \right)^{ 1/p} \Psi^{1/q} (  f) .$$
In virtue of Theorem \ref{theo: spectralcondition},
$$
\Phi^{\G_r} \mbox{{ $(\tilde f$}})=\int_{\G_r  } \tilde f(x) u(x) \times\E_x\int_0^{\tau_{\G_r }} s^l\times (\tilde f  u) (X_s) ds \mu (dx) .
$$
Recall that $r $ is fixed, we do not let tend $r$ to zero. In the same way, 
\begin{multline}
  \Phi^{\bar{\G}^c} \mbox{{ $(\tilde f$}}) = \\
 \int_{\bar{\G}^c  } (f (x) -\mbox{{ $ c $}})(1- u(x)) \times\E_x\int_0^{\tau_{\bar{\G}^c }} s^l\times [(f - \mbox{{ $ c $}}) (1- u)] (X_s) ds \mu (dx) .
\end{multline}
This implies that for bounded $f,$ since $ 0 \le u (.) \le 1,$ and by definition of the constant $c$,
$$\Phi^{\G_r} (\mbox{{ $ \tilde f $}}) \le 
 \| f - c\|^2_\infty \int 1_{\G_r  } (x)\E_x \tau_{\G_r} ^{l+1}\mu (dx) \le Osc(f)^2 \int 1_{\G_r  } (x)\E_x \tau_{\G_r} ^{l+1}\mu (dx) $$
and 
$$ \Phi^{\bar{\G}^c} (\mbox{{ $ \tilde f $}}) \le \| f - c \|_\infty^2 \int 1_{\bar{\G}^c   } (x)\E_x \tau_{\bar{\G}_c} ^{l+1}\mu (dx) 
\le Osc (f)^2 \int 1_{\bar{\G}^c   } (x)\E_x \tau_{\bar{\G}_c} ^{l+1}\mu (dx),$$
since {$ \| f - c \|_\infty \le Osc (f) .$}
This concludes our proof, putting 
$$ \Phi^{1/q} ( f) =  [2  + 2 C(u,r)C_P (\G, r) ]  \Psi^{1/q} (  f).$$
\end{proof}

We give a comment on condition $3.$ of the theorem \ref{theo:oufouf} which shows that basically a non-degeneracy condition on the diffusion like H\"ormander's condition implies the local Poincar\'e inequality. 

\begin{rem}\label{tobecited}
\begin{enumerate}
\item
It is sufficient to replace the local Poincar\'e inequality of condition $3.$ above by: There exists $\Omega \subset \EEE $ a smooth bounded open connected domain such that $\G_r \setminus \bar \G \subset \Omega $ and 
\begin{equation}\label{eq:localpoincarebis}
\int_{ \G_r \setminus \bar \G} f^2 d\mu \le C_P ( \G , r) \int_{\Omega } \Gamma (f, f) d \mu 
\end{equation}
for all $f\in {\cal C} $ having $ \int_{ \G_r \setminus \bar \G} f d \mu = 0 .$
\item 
Wang (2009), \cite{Wang09}, shows that the above local Poincar\'e inequality holds in the following case. Take $ \EEE = \R^d  $ and let  
$$d  \mu = e^V  d \lambda,$$
where $\lambda $ is Lebesgue's measure on $\R^d$ and $V$ smooth and integrable. Let $ X_i, i = 1, \ldots , n ,$ be a family of smooth vector fields satisfying the H\"ormander condition. Consider
$$ A = \sum_{i=1}^n (X_i^2 + (div_\mu X_i) X_i ). $$
Then (\ref{eq:localpoincarebis}) holds. 
\end{enumerate}
\end{rem}

\subsection{Example}
We give more details concerning the example of Remark \ref{tobecited}.2. above and show that in this case all conditions needed for Theorem \ref{theo:oufouf} are satisfied. 
 
Let $\E = \R^d.$ A smooth function is a function belonging to $ C^\infty (\R^d ) ,$ the 
space of all infinitely often differentiable functions from $\R^d $ to $\R .$  A smooth vector field $X$ 
on $\R^d  $ is a linear differential operator 
$ \sum_{ k = 1}^d X^k \partial_k ,$ where $ \partial_k = \frac{\partial}{\partial x^k } $ and where all $X^k $ are smooth functions. We will also 
identify the vector field $X$ with the vector of smooth functions 
$$X = \left( \begin{array}l
X^1 \\
\vdots \\
X^d 
\end{array}
\right) .$$ 
 
In order to define our process, take a family of smooth vector fields $ \{ X_1, \ldots , X_n \} $ satisfying the H\"ormander condition. We recall that this means that for any $x \in \R^d$ there exists $k \geq 2 $ such that the H\"ormander brackets up to order $k$ 
$$ \{ [ X_{i_1} , [ X_{i_2 }, [ \ldots , X_{ i_j } ] \ldots ]] : 2 \le j \le k , 1 \le i_1 , \ldots , i_j \le n \} $$ 
span $\R^d .$ Here, for two smooth vector fields $X$ and $Y,$ the H\"ormander bracket is defined as $[X,Y] = XY - YX,$ the smooth vector field given by the vector of smooth functions 
$$ \left( \begin{array}l
\sum_k (X^k \partial_k Y^1 - Y^k \partial_k X^1)  \\
\vdots \\
\sum_k (X^k \partial_k Y^d  - Y^k \partial_k X^d) 
\end{array}
\right) .$$  
Let $C_c^\infty ( \R^d ) $ be the space of all smooth functions having compact support. 
For any pair of functions $ f, g \in   C_c^\infty ( \R^d ) ,$ define 
$$ \Gamma (f,g) = \sum_{i= 1}^n (X_i f) (X_i g) = \sum_{i=1}^n ( \sum_{k= 1}^d X_i^k \partial_k f ) 
( \sum_{k= 1}^d  X_i^k \partial_k g ) .$$
Let $V$ be a smooth function such that $ e^V \in \mathbb{L}^1  ( \lambda^d ),$ where $\lambda^d $ denotes the $d-$dimensional Lebesgue measure and $ \int e^V d \lambda^d = 1 .$ Put $\mu = e^V \lambda^d ,$ then $\mu$ is a probability measure. $\mu$ will be our reference measure. \\
Now, define an operator $L$ on $\mathbb{L}^2 (\mu ) $ with domain $ \DDD (L) = C_c^\infty (\R^d ) $ by 
$$ L = \sum_{ i = 1}^n ( X_i^2 + ( div_\mu X_i ) X_i ) ,$$
where $ div_\mu X_i = \sum_{k = 1}^d X_i^k \partial_k V + \frac{\partial X_i^k}{\partial x_k } .$

Then $L$ defined on $ \DDD (L)$ is symmetric in $\mathbb{L}^2 ( \mu) ,$ and for all $f,g \in \DDD ( L),$ 
$$ - \int g Lf d \mu = \int \Gamma (f,g) d \mu = {\cal E} ( f, g ) .$$ 

By example 1.3.4 of Bouleau and Hirsch (1991), \cite{BH}, ${\cal E}$ is closable. Let us denote $ (\overline{\cal E} , \overline \DDD ) $ the closure of $ ( {\cal E}, \DDD ( L)) $ and let $ A $ be the generator of $  \overline{\cal E} .$ Then
$ - A$ is a positive self-adjoint extension of $ - L,$ called the Friedrichs extension of $L.$ 
It is standard to show that $ (\overline{\cal E} , \overline \DDD ) $ is a Dirichlet form. 

Assumption \ref{regular} is clearly satisfied, since it is already satisfied for $ ( {\cal E}, \DDD ( L)).$ 
Condition (\ref{eq:constant}) is also satisfied. 

Moreover, $ (\overline{\cal E} , \overline \DDD ) $ is local. This can be seen as follows. 
By remark 5.1.5 of \cite{BH}, it is sufficient to show that for all $f \in C_c^\infty (\R^d ) $ and for $F , G \in C_c^\infty (\R) $ such that $supp (F) \cap supp (G) = \emptyset , $ we have that 
$$ {\cal E} ( F_0 ( f), G_0 (f) ) = 0 ,$$ 
where $ F_0 = F - F(0).$ This follows immediately from 
\begin{multline*}
 \Gamma ( F_0 ( f), G_0 ( f)) = \sum_i \left( \sum_k X_i^k \partial_k ( F\circ f - F(0)) (\sum_k X_i^k \partial_k ( G \circ f - G(0))\right) \\
= \sum_i \sum_k \sum_l X_i^k X_i^l F' (f) \partial_k f \; G' (f) \partial_l f = 0 ,
\end{multline*}
since $ F' \cdot G' \equiv 0$ due to the disjoint supports of $F$ and $G.$ 

Finally, we also have that  $1 \in \overline{\DDD }$ and that $\overline{ \cal E} (1,1) = 0 $ which follows by standard arguments.  

To conclude, all assumptions of Theorem \ref{theo:oufouf} are satisfied except the first one on integrability of hitting times (for the local Poincar\'e inequality, recall remark \ref{tobecited}.2). 

Condition 1. of Theorem \ref{theo:oufouf} is classically implied by a Lyapounov type condition which is related to the rate of divergence of the function { $\nabla V(x) \to - \infty $ }as $|x| \to \infty .$ 
In order to give an explicit example take $d=n$ and $X_i^k=\frac1{\sqrt 2}\delta_i^k.$ In this case
$$Lf=\frac12\Delta f+\frac12\nabla V\nabla f.$$
When identifying with the classical form
$$ Lf = \frac12 \sum_{j,k} a_{j,k} \partial_j \partial_k f + \sum_k b_k \partial_k f $$
we find $b=\frac12\nabla V.$ 
Suppose that for some $r>\frac d2+1+l,$ and $M>0$, 
$$V(x)=-2r\ln |x|,\quad|x|>M.$$ 
Then it is easy to check that Veretennikov's condition (see Veretennikov (1997), \cite{Veret})
$$<b(x),x/|x|>\leq -r/|x|,\quad |x|\geq M,$$
is fulfilled. Under this condition for $A>M,$ $\tau=\inf\{t\geq 0;\ |X_t|\leq A\}$, any  $x\in\R^d$ and some $\varepsilon >0$ $\E_x\tau^{l+1}\leq C(1+|x|^{2l+2+\varepsilon}),$ \cite{Veret} (Theorem 3) and  $\E_{\mu}\tau^{l+1}<\infty.$\\
On the other side, Theorem 3 of Balaji and Ramasubramanian (2000), \cite{indiens}, with $A(x)=1, B(x)=d, C(x)=-2r$ shows that for all $p>r-d/2+1$ and $|x|>A$, $\E_x\tau^{p}=\infty.$

\section{Polynomial moments under Nash inequality.}
In Section \ref{Section3}, we have shown that  for diffusions, the existence of moments implies Nash inequality. 

We now address the inverse question: Does Nash inequality imply the existence of moments? The answer is yes, at least 
if the functional $\Phi $ satisfies (\ref{eq:last}). 

All statements of this section hold true under the general conditions of Section \ref{Section2}, for
a conservative Hunt process which is $\mu-$symmetric, with $\mu$ a probability measure. Let $l > 0 .$ 

\begin{theo}\label{theo:nice}
Suppose that Nash inequality holds with $ p= \frac{l+2}{l+1} $ and with $\Phi $ such that (\ref{eq:last}) holds.  Then for all $\varepsilon > 0 $
and for any open set $\G $ such that $\mu (\G^c) > 0 ,$  
$$ x \to \E_x \tau_{\G}^{l+1 - \varepsilon} \in \mbox{{ $ \mathbb{L}^1  $}} ( \I_{\G } \cdot \mu (dx) ).$$
\end{theo}

The idea of the proof is not new and follows ideas exposed in section 3 of \cite{CG}. 

In the following, $C$ denotes a constant that might change from occurrence to occurrence. For integrable $f$ we write  $\tilde f = f - \mu (f) .$ By \cite{Lig}, we know that under the conditions of Theorem \ref{theo:nice}, 
$$ Var_\mu  ( P_t f ) \le C t^{ - (l+1 )} \Phi (f) \le  C  t^{ - (l+1)}  \mbox{ { $Osc (f)^2$ }} .$$ 
{It can be easily seen that this} implies that the stationary process $X_t $ under $\P_\mu$ is strongly mixing, and by symmetry, its mixing coefficient is bounded by 
$$ \alpha (t) \le C \left( \frac{t}{2}\right)^{- (l+1)}  
= C t^{- (l+1)} .$$

The first step of the proof of Theorem \ref{theo:nice} is the following deviation inequality. 
\begin{prop}
Fix {$t \geq 1   $} and let $V$ be such that $\| V\|_\infty = 1.$ Then 
$$ \P_\mu \left(  \left| \frac1t\int_0^t V(X_s) ds -   \mu(V) \right| \geq 4 \lambda \right) \le  C \left[ \lambda^{-{(l+2)}} \vee \lambda^{ - 2 (l+1)} \right] t^{- (l+1)} .$$
\end{prop}

\begin{proof}
We mimic the proof of Proposition 4.5 of \cite{CG}, trying to loose less by time discretization. First of all, we make use of moment bounds for sums of strongly mixing sequences obtained by Rio in \cite{Rio}. Let $n = [t]$ be the integer part of $t.$ Then 
$$ \int_0^t \tilde V (X_s) ds = \sum_{ k = 1}^n Y_k  , $$
where 
$$ Y_k = \int_{(k-1) t/n}^{k t/n} \tilde V (X_s) ds .$$ 
Then $(Y_j)$ is a $\P_\mu -$stationary sequence of strongly mixing centered random variables, {bounded by $ |Y_j | \le 2  \frac{t}{n},$} with mixing coefficient 
\begin{equation}\label{eq:mixingbar}
\bar \alpha (0 ) = \frac12,\  \bar \alpha (k ) = \alpha \left( (k-1) \frac tn\right) \le C (k-1)^{- (l+1)} \left( \frac tn \right)^{- (l+1)} \le C (k-1)^{- (l+1)} , 
\end{equation}
since $ t/n \geq 1 .$ 
We write $ \tilde Y_j = Y_j / ( 2t/n) $ and apply the inequality (6.19b) of \cite{Rio} to $ S_n = \sum_{k= 1}^n \tilde Y_k ,$ with $ a = l+1 .$ So we obtain for any $ r \geq 1,$ 
\begin{multline*}
 \P_\mu  \left(  \left| \frac1t\int_0^t V(X_s) ds -   \mu(V) \right| \geq 4 \lambda \right) =\\
= \P_\mu \left( \left| \sum_{k=1}^n Y_k\right| \geq 4 \lambda t \right ) 
= \P_\mu \left( | S_n| \geq 4 (\lambda n /2)   \right ) \\
\le 4 \left( 1 + \frac{\lambda^2 n^2}{ 4 r s_n^2 } \right)^{- r/2}+ 
2 n c r^{-1} \left( \frac{4 r}{\lambda n }\right)^{l+2} .
\end{multline*} 
Here, 
$$ s_n^2 = \sum_{i=1}^n \sum_{j = 1}^n | Cov (\tilde Y_i, \tilde Y_j) | .$$ 
We have to control this {sum of covariances} $s_n^2.$
Using Corollaire 1.1. of \cite {Rio} we have
$$s_n^2\leq 4\sum_{k=1}^n\int_0^1[\alpha^{-1}(u) \wedge n ] Q_k^2(u)du,$$
where 
$$a^{-1}(u)=\inf\{k\in\N;\; \bar{\alpha}(k)\leq u\}=\sum_{i\geq 0}\1_{\bar{\alpha}(i)>u}, \quad a^{-1}(u) \wedge n = \sum_{i=0}^{n-1} \1_{\bar{\alpha}(i)>u},$$ 
and $Q_k(u)$ is the inverse function of $H_{\tilde {Y}_k}(t)=\P(|\tilde {Y}_k|>t).$
Since $|\tilde Y_k|\leq 1$ for all $k\leq n,$  
$Q_k^2(u)\leq 1$ and thus (see \cite{Rio}, page 15),
$$s_n^2\leq 4n\int_0^1[a^{-1}(u)\wedge n]  du\leq 4 n \sum_{i=0}^\infty \bar{\alpha}(i) .$$
Since $l>0, $ this last series converges (compare to (\ref{eq:mixingbar})), and we obtain
$$ s_n^2 \le C  n $$
for some constant $C > 0 .$ So finally, 
\begin{multline*}
 \P_\mu  \left(  | \frac1t \int_0^t V(X_s) ds -   \mu(V) | \geq 4 \lambda \right) 
\le\\
\le  4 \left( 1 + \frac{\lambda^2 n^2}{4 C r  n  } \right)^{- r/2} + 
2 n c r^{-1} ( \frac{4 r}{\lambda n })^{l+2}\\
\le  4 \left(  \frac{\lambda^2 n}{C r    } \right)^{- r/2} + 
2  c r^{-1} ( \frac{4 r}{\lambda  })^{l+2} n^{ - (l+1)}  .
\end{multline*}
Finally we choose $r = 2 (l+ 1)$ and use that 
$$ n^{- (l+1) } = \left( \frac{t}{n}\right)^{ l+1} t^{- (l+1)} \le 2^{l+1} t^{ - (l+1)} ,$$
for $ t \geq 1, $ since $ n = [t] $ and $t \geq 1 $ which implies $ t/n \le 2 .$ Thus we get the result.  
\end{proof}

\noindent
{\it Proof of Theorem \ref{theo:nice}}
We apply the above deviation inequality with $V = 1_{ \G^c} .$ Then 
$$\{ \tau_{\G } > t \} \subset \left\{ \frac1t \int_0^t V(X_s) ds = 0\right\} 
\subset \left\{  \frac1t | \int_0^t \tilde V(X_s) ds  | \geq \mu( \G^c) \right\} .$$
Hence
$$ \int \!\! 1_{ \G} (x) \mu (dx) \P_x ( \tau_{\G } > t ) 
\le \P_\mu \left( \frac1t | \int_0^t \tilde V(X_s) ds  | \geq \mu( \G^c)  \right) 
\le C t^{ - (l+1)} ,$$
whence for every $\varepsilon > 0 $ ``small''  
$$ \E_x \tau_{\G}^{l+1 - \varepsilon} \in \mbox{{ $ \mathbb{L}^1 $}} ( \I_{\G} \cdot \mu (dx)).$$

\def\refname{References}

\end{document}